\documentclass[reqno,tbtags,intlimits,a4paper,oneside,12pt]{amsart}
\usepackage[cp1251]{inputenc}%
\usepackage[english]{babel}
\usepackage{amssymb,upref,calc,url}
\usepackage[mathscr]{eucal}
\usepackage{graphicx}
\usepackage{amsmath,amscd,amsthm,verbatim}
\usepackage[all]{xy}
\usepackage[in]{fullpage}
\newtheorem{thm}{Theorem}[section]
\newtheorem*{thm*}{Theorem}
\newtheorem{lem}[thm]{Lemma}

\newtheorem{cor}[thm]{Corollary}

\newtheorem{conj}[thm]{Conjecture}

\theoremstyle{definition}
\newtheorem{dfn}[thm]{Definition}

\DeclareMathOperator{\tgh}{th}
\DeclareMathOperator{\cth}{cth}
\DeclareMathOperator{\sh}{sh}

\DeclareMathOperator{\R}{Re}
\DeclareMathOperator{\I}{Im}

\begin{document}

\title{Hirzebruch Functional Equation:\\
Classification of Solutions}
\author{Elena~Yu.~Bunkova}
\address{Steklov Mathematical Institute of Russian Academy of Sciences, Moscow, Russia.}
\email{bunkova@mi.ras.ru}
\thanks{This work is supported by the Russian Science Foundation under grant 14-50-00005.}

\begin{abstract}
The Hirzebruch functional equation is
\[
	\sum_{i = 1}^{n} \prod_{j \ne i} { 1 \over f(z_j - z_i)} = c
\]
with constant $c$ and initial conditions $f(0)=0, f'(0)=1$.
In this paper we find all solutions of the Hirzebruch functional equation for~$n \leqslant 6$ in the class of meromorphic 
functions and in the class of series.
Previously, such results were known only for $n \leqslant 4$. 

The Todd function is the function determining the two-parametric Todd genus (i.e. the~$\chi_{a,b}$ genus). It gives a solution to the Hirzebruch functional equation for any $n$.
The elliptic function of level~$N$ is the function determining the elliptic genus of level~$N$.
It~gives a solution to the Hirzebruch functional equation for $n$ divisible by $N$.

A series corresponding to a meromorphic function $f$ with parameters in $U \subset \mathbb{C}^k$
is a~series with parameters in the Zariski closure of~$U$ in~$\mathbb{C}^k$, such that for parameters in~$U$ it coincides with the series expansion at zero of $f$.
The main results are:

\begin{thm*}{\bf \ref{T62}.}
Any series solution of the Hirzebruch functional equation for $n = 5$
corresponds to the Todd function or to the elliptic function of level~$5$.
\end{thm*}

\begin{thm*}{\bf \ref{Tn6c1}.}
Any series solution of the Hirzebruch functional equation for $n = 6$
corresponds to the Todd function or to the elliptic function of level $2$, $3$ or $6$.
\end{thm*}

This gives a complete classification of complex genera that are fiber multiplicative with respect to $\mathbb{C}P^{n-1}$ for $n \leqslant 6$.
\end{abstract}

\maketitle

\vspace{-10pt}

\section{Introduction}

The problem solved in this work originates in the theory of Hirzebruch genera. 

The Hirzebruch genus is one of the most important classes of invariants of manifolds.
A series~$f(z) = z + \sum_{k=1}^{\infty} f_k z^{k+1}$ with~$f_k$ in a ring $R$ determines
a Hirzebruch genus of stably complex manifolds (see \cite{Hir} and \cite[Section E.3]{T}).
The condition for a~complex genus to be fiber multiplicative with respect to $\mathbb{C}P^{n-1}$
is given by the Hirzebruch functional equation in $f(z)$ 
(see \cite[Chapter 4]{Hirtz}, \cite[Chapter 9]{T}, and \cite[Section 4]{BN}).

The \emph{Hirzebruch functional equation} is 
\begin{equation} \label{Hfe} 
	\sum_{i = 1}^{n} \prod_{j \ne i} { 1 \over f(z_j - z_i)} = c.
\end{equation}
Here $n$ is a natural number greater than $1$ and $c$ is a constant.
A function $f(z)$ is a solution if \eqref{Hfe} holds whenever it's left hand side is defined.
We consider this equation in the class of meromorphic functions $f(z)$ with initial conditions $f(0) = 0$, $f'(0) = 1$.
We~also work in the class of series solutions, i.e.
$f(z) = z + \sum_{k=1}^{\infty} f_k z^{k+1}$ with~$f_k \in \mathbb{C}$.
For a meromorphic function $f(z)$ take it's series expansion at zero to get this series.

In this work we solve the problem of classification of Hirzebruch functional equation~\eqref{Hfe} solutions for $n \leqslant 6$.
For $3 \leqslant n \leqslant 6$ we show that all the solutions determine either the two-parametric Todd genus
(that is, the $\chi_{a,b}$ genus)
or the elliptic~genus of level~$N$ with~$N \mid n$. For $N = 2$ this is the famous Ochanine--Witten genus \cite{Osh, Wit}.

The fact that the functions determining this genera give solutions of \eqref{Hfe} is well-known (see \cite{Krichever-90, T}).
Our classification result states that there are no other solutions.

An immediate corollary is that these are all the complex genera that are
fiber multiplicative with respect to~$\mathbb{C}P^{n-1}$ for $n < 6$. The constant $c$ is the value of the genus on~$\mathbb{C}P^{n-1}$.
To~classify the solutions of \eqref{Hfe} we use methods of elliptic functions theory, differential equations theory, complex analysis, and algebraic geometry.

In the case $n = 2$ equation \eqref{Hfe} takes the form $(1 / f(z_2 - z_1)) + (1 / f(z_1 - z_2)) = c. $  
It's~general analytic solution is $f(z) = 2 z / (c z + 2 g(z^2))$ for any function $g(y)$ regular at~$y = 0$ and such that $g(0) = 1$.
This proved to be the only infinite-dimensional case~\cite{Man}.

The case $n = 3$ is solved in \cite{BBZam}.
The case $n = 4$ is solved in \cite{BN} with a closer inspection of the subcase $c = 0$ in~\cite{Ell4}.
This works describe coefficients of elliptic functions of level~$3$ and $4$.
In \cite{Oeshen} such a description was given in terms of Jacobi polynomials.

The cases $n = 5$ and $6$ solved here are new.
The work is organized as~follows:

In Section \ref{Sect1} we re-prove that the Todd function (the function determining the two-parametric Todd genus) and the elliptic function of level $N$
for  $N \mid n$ give solutions~of~\eqref{Hfe}.

In Section \ref{Sect2} we describe the manifold of series solutions of the Hirzebruch functional equation \eqref{Hfe}.
In~\cite{Man} we show that for~$n > 2$ it is an algebraic manifold in~$\mathbb{C}^{n}$. 
This result is a crucial part of the classification results that we obtain in Sections \ref{Clas3}--\ref{Clas6}.
The coordinates for this manifold are the first coefficients of the $Q$-series $Q(z) = z/f(z)$.

The Baker--Akhiezer function \cite{Krichever-90, BA} plays a central role in this work.
All the genera discussed are special cases of the Krichever genus \cite[Section E.5]{T}.
It is determined by a function~$\varphi(z)$ with $\varphi(z) \Phi(z;\rho) = \exp(\alpha z)$,
where $\Phi(z;\rho)$ is the Baker--Akhiezer function. We give necessary information on the function~$\varphi(z)$ in Section \ref{Sect3}.

The work \cite{Ell4} gives a family of differential equations solved by $\varphi(z)$.
In Section \ref{Sect4} we use it to parametrize series solutions to Hirzebruch functional equation \eqref{Hfe}.
It is still a conjecture that all solutions of the Hirzebruch functional equation for all~$n$ and given initial conditions satisfy this family.

In Sections \ref{Sect5} and \ref{Sect6} we identify the Todd function and the elliptic function of level $N$
with two-parametric families of solutions of the differential equation from Section~\ref{Sect4}.
In Sections~\ref{Clas5} and~\ref{Clas6} we give descriptions of the coefficients of elliptic functions of~level~$5$~and~$6$ 
in terms of this differential equation.
The~differential equation we use does not depend on~$N$ (cf. \cite[Appendix III, eq.~(37)]{Hirtz}).
It is a corollary of the functional equation from~\cite{BFE}.

In Sections \ref{Clas3}-\ref{Clas6} we classify the solutions of the Hirzebruch functional equation~\eqref{Hfe} for~$n = 3,4,5,6$ 
using the ingredients from Sections \ref{Sect1}-\ref{Sect6}.

The functional equation studied in this work is one of the equations arising in \cite{B18}.
Related works include classification results for systems of equations \eqref{Hfe} taken for different $n$ simultaneously (see \cite[Section 4.6]{Hirtz} and \cite{OM}).
Other topics in the intersection of Baker--Akhiezer functions and functional equations include
\cite{BBr, BrF, BFE, BV}, and Lax equations~\cite{KL}.
Recent developments of ideas of \cite{KL} are given in \cite{S1}-\cite{S4}.

In applications to Hirzebruch genera the ring $R$ plays an important role.
Nevertheless, in this work we set~$R = \mathbb{C}$ to take advantage of complex analysis.
A~classical means of study of the coefficients rings $R$ are formal groups. 
This approach originates from~\cite{GC}.
The properties of a formal group corresponding to Krichever genus are described in~\cite{BFE}.
The rings of coefficients of elliptic genera of level $2$ and $3$ are described~in~\cite{BU, BBU}.

\section{Todd functions and elliptic functions solutions} \label{Sect1}

In this section we re-prove known theorems (see \cite[Section 4.6]{Hirtz}) that follow from
the corresponding properties of Hirzebruch genera (see \cite[Chapter 9]{T}):
the two-parametric Todd genus ($\chi_{a,b}$-genus, see \cite[Example E.3.9]{T}) and the elliptic genus of level $N$ (see~\cite{Krichever-90}).

\begin{thm} \label{tft}
The function \vspace{-3pt}
 \begin{equation} \label{ft}
 f(z) = {e^{a z} - e^{b z} \over a e^{b z} - b e^{a z}}.
 \end{equation}
is a solution of the Hirzebruch functional equation \eqref{Hfe} for any $n$ and $(-1)^n c = - {a^n - b^n \over a - b}$.
\end{thm}

The function \eqref{ft} determines the two-parametric Todd genus. We call it \emph{Todd function}. 

\begin{proof}
For \eqref{ft} we have \vspace{-3pt}
\[
{1 \over f(z)} = {a e^{b z} - b e^{a z} \over e^{a z} - e^{b z}} = - {a + b \over 2} + {a - b \over 2} \cth \left({a - b \over 2} z\right).
\]

For a set of points $z_j$, $j = 1, ..., n$, with $z_i \ne z_j$ for $i \ne j$ and $0<\I(z_j)<\pi I$ for all $j$, consider the function \vspace{-5pt}
\[ 
 F(z) = \prod_{j = 1}^{n} \cth(z - z_j).
\]
We have $F(z+\pi I) = F(z)$ and $F(z)$ has simple poles at $z = z_j$. Set $R > |\R(z_i)|$.
The~sum of residues of $F(z)$ in the strip $0\leqslant \I(z) \leqslant\pi$ by Cauchy formula is equal to 
\[
 \sum_{i = 1}^{n} \prod_{j \ne i} \cth(z_i - z_j) = {1 \over 2 \pi I} \int_{R}^{R + \pi I} F(z) dz - {1 \over 2 \pi I} \int_{- R}^{- R + \pi I} F(z) dz.
\]
For $\R(y) \to \infty$ we have $F(y) \to 1$ and $F(-y) \to (-1)^n$. Therefore
\begin{equation} \label{tg}
	\sum_{i = 1}^{n} \prod_{j \ne i} \cth(z_j - z_i) = \left\{ \begin{matrix}
	                                                           0, \quad \text{for} \quad n \quad \text{even},\\
	                                                           1, \quad \text{for} \quad n \quad \text{odd}
	                                                          \end{matrix}
\right.
\end{equation}
and $f(z) = \tgh(z)$ is a solution of equation \eqref{Hfe} for all $n$ and $c$ given by \eqref{tg}.
Similarly, 
\begin{equation} \label{itg}
	\sum_{i = 1}^{n} \prod_{j \ne i} s \cth \left(s (z_j - z_i)\right) = \left\{ \begin{matrix}
	                                                           0, \quad \text{for} \quad n \quad \text{even},\\
	                                                           s^{n-1}, \quad \text{for} \quad n \quad \text{odd}
	                                                          \end{matrix}
\right.
\end{equation}
and $f(z) = \tgh(s z)/s$ is a solution of equation \eqref{Hfe} for all $n$ and $c$ given by \eqref{itg}. Finally, \vspace{-3pt}
\begin{multline*}
\sum_{i = 1}^{n} \prod_{j \ne i} \left(- {a + b \over 2} + s \cth \left(s (z_j - z_i)\right)\right) = \\
= \sum_{i = 1}^{n} \sum_{k = 0}^{n-1} \left(- {a + b \over 2}\right)^k \sum_{(s_1, s_2 \ldots, s_k) \not\ni i}
\prod_{j \ne i, j \ne s_1, \ldots, j \ne s_k} s \cth \left(s (z_j - z_i)\right) = \\
= \sum_{k = 0}^{n-1} \left(- {a + b \over 2}\right)^k \sum_{(s_1, s_2 \ldots, s_k)} \left( \sum_{i \notin (s_1, s_2 \ldots, s_k)}
\prod_{j \ne i, j \ne s_1, \ldots, j \ne s_k}s \cth \left(s (z_j - z_i)\right) \right).
\end{multline*}
We apply  \eqref{itg} and set $2 s = a - b$.
The combinatorial formula \vspace{-3pt}
\[
 \sum_{k = 0}^{{n-1 \over 2}} C_{n}^{1+2k} \left(- {a + b \over 2}\right)^{n-1-2k} \left({a - b \over 2}\right)^{2k} = (-1)^{n-1} {a^n - b^n \over a - b}
\]
finishes the proof. \end{proof}

\eject

Let $L$ be a lattice in $\mathbb{C}$.

\begin{dfn}[{\cite[Appendix III, Section 1]{Hirtz}}] \label{d12}
An \emph{elliptic function of level $N$} with lattice~$L$ is a meromorphic function~$f$ such that
$f(0) = 0$, $f'(0) = 1$,
and $g(z) = f(z)^N$ is an elliptic function with~lattice $L$ and divisor $N \cdot 0 - N \cdot \rho$ for $\rho \in \mathbb{C}$.
Additionally, we ask $N \in \mathbb{N}$ to be the~minimal number for $f(z)$ with such property.
\end{dfn}

From Definition \ref{d12} if follows that $\rho$ is an $N$-division point of the lattice, that is $\rho \notin L$, $N \rho \in L$.
Moreover, given $L$ and an $N$-division point $\rho$, one can construct $g(z)$ and~$f(z)$ as in Definition \ref{d12} uniquely (see \cite{Hirtz}).
The minimality condition implies the order of $\rho$ as an~element of $\mathbb{C}/L$ is $N$.
The function $f(z)$ is itself elliptic with respect to a sublattice $L'$ of~$L$ of~order $N$.

For a lattice $L$ there are $(N^2 - 1)$ $N$-division points up to addition of points in $L$.
Thus for a fixed lattice there are $(N^2 - 1)$ elliptic functions of level~$k \mid N$.

The following Lemma assigns to a basis $\omega, \omega'$ in $L$ an elliptic function,
so we can speak of ``the'' elliptic function of level~$N$ with parameters $\omega, \omega'$.
The elliptic function of level~$N$ determines the elliptic genus of level~$N$ (see \cite[Appendix III, Section 3]{Hirtz}). 

\begin{lem} \label{llN}
For $f(z)$ an elliptic function of level $N$ with lattice~$L$ one can choose the generators $\omega, \omega'$ of $L$ such that 
$\rho = \omega / N$ and the periodicity properties hold
\begin{align}
f(z + \omega) &= f(z), & f(z + \omega') &= \exp\left(- 2 \pi I / N\right) f(z). 
\end{align}
\end{lem}

\begin{proof}
 Let $L = \langle \upsilon, \upsilon' \rangle$. As $g(z) = f(z)^N$ is an elliptic function with lattice $L$,
 we have 
 \begin{align}
 f(z + \upsilon) &= \varepsilon_1 f(z), & f(z + \upsilon') &= \varepsilon_2 f(z), \label{ups}
 \end{align}
 where $\varepsilon_1^N = \varepsilon_2^N = 1$. For any prime root $\varepsilon$ of $1$ of degree $N$, let
 $\varepsilon_1 = \varepsilon^{s_1}$, $\varepsilon_2 = \varepsilon^{s_2}$.
 By~changing the basis of $L$ using Euclid's algorithm for $(s_1, s_2)$, we obtain $L = \langle \omega, \omega'\rangle$~with
  \begin{align}
 f(z + \omega) &= f(z), & f(z + \omega') &= \varepsilon f(z). \label{star}
 \end{align}
Here we use the minimality condition implying that $s_1$ and $s_2$ are coprime.

Now we show that $\rho = {k \over N} \omega$, where $k$ and $N$ are coprime. As $\rho$ is defined up to addition of $\omega$ and $\omega'$,
let us take a representative for $\rho$ in the parallelogram spanned by $\omega$ and $\omega'$.
As $\rho$ is an $N$-division point of the lattice, we have $\rho = {k \over N} \omega + {k' \over N} \omega'$ for some $k, k'$ between~$0$ and $N-1$.
The function $f(z)$ is elliptic with respect to the sublattice $L' = \langle \omega, N \omega' \rangle$
and has simple poles in points $\rho, \rho + \omega', \rho + 2 \omega', \ldots, \rho + (N-1) \omega'$ and simple zeros
in~$0, \omega', 2 \omega', \ldots, (N-1) \omega'$.
The sum of affixes of the poles minus the sum of affixes of the zeros is a period (see \cite[Section 20.14]{WW}),
thus $N \rho = k \omega + k' \omega' \in L'$ and $k' = 0$.

We have obtained that to each prime root $\varepsilon$ of $1$ of degree $N$ corresponds a point $\rho = {k \over N} \omega$ for some $k$ with $(k,N) = 1$.
For different prime roots $\varepsilon$ the points $\rho$ are different:
if not, for two elliptic functions $f_1(z)$ and $f_2(z)$ of level $N$ with the same $\rho$ the elliptic function $f_1(z)/f_2(z)$ has no poles and thus is constant.

Using $\varepsilon$ that corresponds to $\rho = \omega/N$ in the construction above, we obtain the generators $\omega, \omega'$ of $L$
and periodicity properties \eqref{star}.

The concrete form for $\varepsilon$ will follow from results of \cite{Krichever-90} that we describe in Section \ref{Sect3}.
Setting $\rho = \omega/N$ in \eqref{fKr}, using \eqref{per} for $\omega$ and $\omega'$
and the Legendre identity 
\begin{equation}\label{Leg}
\zeta(\omega / 2) \omega' - \zeta(\omega' / 2) \omega = \pi I,
\end{equation}
we get $\alpha = \zeta(\omega/N) - 2 \zeta(\omega/2)/N$ and $\varepsilon = \exp\left(- 2 \pi I / N\right)$.
The Legendre identity assumes that $\I(\omega'/\omega)>1$. This can be obtained by a basis change $(\omega, \omega') \to (\omega, -\omega')$.
\end{proof}

\begin{cor}[\cite{Krichever-90}] \label{CorKri}
The elliptic function of level $N$ with parameters $\omega, \omega'$
is given by~\eqref{fKr} with lattice $L = \langle \omega, \omega' \rangle$ and $\alpha = \zeta(\omega/N) - 2 \zeta(\omega/2)/N$, $\rho = \omega/N$.
\end{cor}

\eject

\begin{thm} \label{tn}
The elliptic function of level $N$
is a solution of the Hirzebruch functional equation \eqref{Hfe} for any $n$ such that $N \mid n$ and $c = 0$.
\end{thm}

\begin{proof}
For $f(z)$ an elliptic function of level $N$ with lattice $L = \langle \omega, \omega' \rangle$
and a set of points~$z_j$, $j = 1, ..., n$,
with $z_j$ in the period parallelogram spanned by~$\langle \omega, \omega' \rangle$
and $z_i \ne z_j$ for $i \ne j$, consider the function

\vspace{-14pt}

\[
 F(z) = \prod_{j = 1}^{n} {1 \over f(z - z_j)}.
\]
This function has simple poles at $z = z_j$ and no other poles in the period parallelogram.
By Lemma \ref{llN} for $N \mid n$ the function $F(z)$ is elliptic with lattice $L$.

The~sum of residues of $F(z)$ in the period parallelogram is
\[
 \sum_{i = 1}^{n} \prod_{j \ne i} {1 \over f(z_i - z_j)}, 
\] 
and by Cauchy formula it is equal to zero.

Note that for $z_j$ outside the period parallelogram one can take a representative $z_j + m \omega + n \omega'$ and using \eqref{star} obtain the same conclusion.
For $z_i = z_j \mod L$ the left hand side of \eqref{Hfe} is not defined. 
\end{proof}

\section{Manifold of series solutions} \label{Sect2}

Consider the Hirzebruch functional equation \eqref{Hfe} 
\[
	\sum_{i = 1}^{n} \prod_{j \ne i} { 1 \over f(z_j - z_i)} = c.
\]
It's solutions in the class of formal series $f(z) = z + \sum_{k=1}^\infty f_k z^{k+1}$ with $f_k$ in some ring $R$ are called formal solutions.
For this work we set $R = \mathbb{C}$.

For each solution $f(z)$ of the Hirzebruch functional equation \eqref{Hfe} in the class of meromorphic functions with $f(0) = 0$, $f'(0) = 1$,
consider it's series expansion at zero.
This gives a series solution $f(z) = z + \sum_{k=1}^\infty f_k z^{k+1}$ with $f_k \in \mathbb{C}$.

Set $Q(z) f(z) = z$ and 
$
Q(z) = 1 + \sum_{k=1}^\infty q_k z^k
$
with
$q_k \in \mathbb{C}$.
There are obvious polynomial expressions for $f_k$ in $q_1, q_2, \ldots, q_k$ that are linear in $q_k$.
The problem of finding solutions the class of series for $f(z)$ is equivalent to this problem for $Q(z)$.
For convenience we use the coefficients $q_k$ to parametrize the solutions $f(z)$ of \eqref{Hfe}.

In $Q(z)$ equation~\eqref{Hfe} takes the form
\begin{equation} \label{fe2} 
\sum_{i = 1}^{n} \prod_{j \ne i} { Q(z_j - z_i) \over (z_j - z_i)} = c.
\end{equation}

\begin{thm}[{\cite[Theorem 19]{Man}}] \label{TM}
The manifold of parameters of the universal formal solution of \eqref{Hfe}
for $n > 2$ is an algebraic manifold $\mathcal{M}_n$ in~$\mathbb{C}^{n}$ with coordinates $q_1, \ldots, q_{n}$.
\end{thm}

In our case it follows that for any $(q_1, \ldots, q_{n}) \in \mathcal{M}_n$ there exists a unique series~$Q(z)$ with this initial terms that solves \eqref{fe2}
for some $c$,
and for every solution~$Q(z)$ of \eqref{fe2} we have $(q_1, \ldots, q_{n}) \in \mathcal{M}_n$.
We will consider the space $\mathbb{C}^{n}$ as the space of parameters of series $f(z)$
with coordinates $q_1, \ldots, q_{n}$ and Zariski topology.

\begin{cor} \label{corU}
Suppose that for some subset $U$ of $\mathbb{C}^{n}$ the series with parameters in $U$ are solutions of the Hirzebruch functional equation \eqref{Hfe}.
Then the series with parameters in the Zariski closure of $U$ in $\mathbb{C}^{n}$ are solutions to the Hirzebruch functional equation.
\end{cor}
\begin{proof}
The algebraic manifold $\mathcal{M}_n$ is closed, so if $U \in \mathcal{M}_n$, then $\bar{U} \in \mathcal{M}_n$.
\end{proof}

\eject

The problem of Hirzebruch functional equation \eqref{Hfe} 
solutions  classification for $f(z)$ in the class of series $f(z) = z + \sum_{k=1}^\infty f_k z^{k+1}$ with $f_k \in \mathbb{C}$ thus consists of two parts:
\begin{itemize}
 \item For each $n$ describe the algebraic manifold $\mathcal{M}_n$.
 \item For each $(q_1, \ldots, q_{n}) \in \mathcal{M}_n$ describe a series solution of \eqref{Hfe}.
\end{itemize}

In the following sections we solve both parts for $3 \leqslant n \leqslant 6$.

\section{Krichever function} \label{Sect3} 

Let $\sigma$, $\zeta$, $\wp$ denote Weierstrass functions \cite{Iso, WW} with lattice $L$.

In this section we work with the function from \cite{Krichever-90} that we call the \emph{Krichever function} \vspace{-10pt}
\begin{equation} \label{fKr} 
\varphi(z) =  { \sigma(z) \sigma(\rho) \over \sigma(\rho - z)} \exp(\alpha z - \zeta(\rho) z).
\end{equation}

The function~\eqref{fKr} determines the Krichever genus (see \cite[Section E.5]{T}, and \cite{Krichever-90}).

The periodicity properties for \eqref{fKr} follow from periodicity properties for Weierstrass functions (or see \cite[equation (7')]{BA}).
For any $\omega$ in a basis of $L$ we have \vspace{-2pt}
\begin{equation} \label{per}
\varphi(z + \omega) = \exp(\alpha \omega + 2 \zeta(\omega / 2) \rho - \zeta(\rho) \omega ) \varphi(z). \vspace{-2pt}
\end{equation}

The Krichever function \eqref{fKr} depends on the parameters $\alpha$, $\rho$, and the lattice $L$.
The theory of Weierstrass functions allows to relate $L$ with the parameters $g_2$, $g_3$ of an elliptic curve, namely,
the Weierstrass equation is \vspace{-2pt}
\begin{equation} \label{We}
 \wp'(z)^2 = 4 \wp(z)^3 - g_2 \wp(z) - g_3, \vspace{-2pt}
\end{equation}
where $\wp$ is the Weierstrass function with lattice $L$.
Given $L$ one constructs the Weierstrass function with this lattice and determines $g_2, g_3$ from \eqref{We}.
The converse is also true, as the differential equation \eqref{We} with parameters $g_2, g_3$ gives as solution a Weierstrass function with lattice $L$.
The parameter $\rho$ is determined by the pair $\wp(\rho)$, $\wp'(\rho)$ up to addition of a point in $L$.
As \eqref{We} holds for $z = \rho$, it determines $g_3$ given $\wp(\rho)$, $\wp'(\rho), g_2$.
Therefore for the Krichever function \eqref{fKr} we can take $(\alpha, \wp(\rho), \wp'(\rho), g_2)$ as parameters.

The construction above does not work for some values of parameters. Let us specify this restrictions.
We have $\alpha \in \mathbb{C}$, $\rho \in \mathbb{C} \backslash L$ and $L$ is a (non-degenerate) lattice.
The parameters $g_2, g_3$ correspond to a non-degenerate lattice whenever $g_2^3 - 27 g_3^2 \ne 0$.
For such $g_2$, $g_3$ we get $\wp$, and any values for $\wp(\rho)$, $\wp'(\rho)$ with \eqref{We} determine a point 
$\rho \in \mathbb{C} \backslash L$ up to addition of $L$.
Thus for any set of parameters $(\alpha, \wp(\rho), \wp'(\rho), g_2) \in \mathbb{C}^4$
such that for $g_3$ determined by \eqref{We} with $z = \rho$ holds $g_2^3 - 27 g_3^2 \ne 0$ we obtain a Krichever function~\eqref{fKr}.

For $g_2^3 - 27 g_3^2 = 0$ equation \eqref{We} determines a function that is not elliptic.
In this case the construction for $\varphi(z)$ works for functions $\sigma(z)$ and $\zeta(z)$ with parameters $g_2, g_3$ with no modifications. 
The difference is that we don't get a lattice $L$ with \eqref{per}.
As this lattice is crucial for Definition \ref{d12} (see Corollary \ref{CorKri}), 
further in this work we suppose that for~\eqref{fKr} holds $g_2^3 - 27 g_3^2 \ne 0$. 

The function \eqref{fKr} has the series expansion at zero \cite[equation (E.29)]{T} \vspace{-8pt}
\begin{multline} \label{dfKr} 
\varphi(z) = z + \alpha z^2 + (\alpha^2 + \wp(\rho)) {z^3 \over 2}
+ (\alpha^3 + 3 \alpha \wp(\rho) - \wp'(\rho)) {z^4 \over 3!} + \\+ 
(\alpha^4  + 6 \alpha^2 \wp(\rho) + 9 \wp(\rho)^2 - 4 \alpha \wp'(\rho) - 3 g_2/5) {z^5 \over 4!} + O(z^6).
\end{multline}

Set $Q(z) \varphi(z) = z$ and 
$
Q(z) = 1 + \sum_{k=1}^\infty q_k z^k.
$
The expansion \eqref{dfKr} gives the relations between $(\alpha, \wp(\rho), \wp'(\rho), g_2)$ and $(q_1, q_2, q_3, q_4)$:
\begin{align*}
\alpha &= - q_1, & \wp(\rho) &= q_1^2 - 2 q_2, & \wp'(\rho) &= 2 (q_1^3 - 3 q_1 q_2 + 3 q_3), & g_2 &=  20 (q_2^2 - 2 q_1 q_3 + 2 q_4).
\end{align*}

Further we will use the set $(q_1, q_2, q_3, q_4)$ as parameters. The condition $g_2^3 - 27 g_3^2 \ne 0$ is
\begin{equation} \label{det}
2^2 \cdot 5^3 (q_2^2 - 2 q_1 q_3 + 2 q_4)^3 \ne 3^3 (2 q_1^2 q_2^2 - 2 q_2^3 - 4 q_1^3 q_3 + 2 q_1 q_2 q_3 + 9 q_3^2 + 10 q_1^2 q_4 - 20 q_2 q_4)^2. 
\end{equation}

\eject

\section{Family of differential equations} \label{Sect4}
Consider a family of differential equations
\begin{equation} \label{feq}
f(z) f'''(z) - 3 f'(z) f''(z) = 6 q_1 f'(z)^2 + 12 q_2 f(z) f'(z) + 12 q_3 f(z)^2
\end{equation}
with parameters $(q_1, q_2, q_3) \in \mathbb{C}^3$.
We say that $f(z)$ is a solution of \eqref{feq} if it satisfies the differential equation for some parameters.
We consider solutions to \eqref{feq} in the class of meromorphic functions with $f(0) = 0$, $f'(0) = 1$.
A series expansion of $f(z)$ at zero gives a solution in the class of series $f(z) = z + \sum_{k=1}^\infty f_k z^{k+1}$ with $f_k \in \mathbb{C}$.

Set $Q(z) f(z) = z$ and 
$
Q(z) = 1 + \sum_{k=1}^\infty q_k z^k
$
with
$q_k \in \mathbb{C}$.
For convenience we use the coefficients $q_k$ to parametrize the solutions $f(z)$ of \eqref{feq}.
A direct substitution shows that the coefficients $(q_1, q_2, q_3)$ of any solution coincide with the parameters $(q_1, q_2, q_3)$ of~\eqref{feq}.
We have $f_1 = - q_1$, $f_2 = q_1^2 - q_2$, $f_3 = -q_1^3 + 2 q_1 q_2- q_3$, $f_4 = q_1^4 - 3 q_1^2 q_2 + 2 q_1 q_3 + q_2^2 - q_4$.

\begin{lem}[{\cite[Corollary 2.3]{Ell4}}]\label{L2}
The Krichever function \eqref{fKr} is a solution of \eqref{feq}.
\end{lem}

\begin{lem}[{\cite[Lemma 3.3]{Ell4}}] \label{L3}
Two solutions of \eqref{feq} with $f(0) = 0$, $f'(0) = 1$ coincide,
if initial terms of their expansions as power series in $z$ at zero
up to $z^5$ coincide.
\end{lem}

\begin{cor}
For any $(q_1, q_2, q_3, q_4) \in \mathbb{C}^4$ there is a unique series solution of \eqref{feq}. 
\end{cor}

\begin{proof}
The substitution of $f(z) = z + \sum_{k=1}^\infty f_k z^{k+1}$ with $f_k$ expressed in $q_k$ into \eqref{feq}
gives a countable number of polynomial relations on~$q_k$. 
By Lemma \ref{L3}, any series solution to~\eqref{feq} is determined by $(q_1, q_2, q_3, q_4) \in \mathbb{C}^4$. 
Therefore, from the polynomial relations follow expressions for $q_k$ in $(q_1, q_2, q_3, q_4)$.
If there are any relations left, they become polynomial relations in $(q_1, q_2, q_3, q_4)$.
By Lemma \ref{L2}, for any $(q_1, q_2, q_3, q_4) \in \mathbb{C}$ with \eqref{det} there is a solution to \eqref{feq}.
Therefore there are no relations in $(q_1, q_2, q_3, q_4)$.
\end{proof}

By Lemma~\ref{L2}, for a Zariski open subset of parameters given by \eqref{det} the series solution coincides with the series expansion at zero
of the Krichever function with this parameters. A solution outside this Zariski open subset is what we call the \emph{singular Krichever function}
\begin{equation} \label{fKrs} 
\varphi_s(z) =  \exp(\alpha z - \kappa z) / (\eta \cth(\eta z) - \kappa).
\end{equation}
We have $\varphi_s(z) \Phi_s(z;\eta) = \exp(\alpha z)$, where $\Phi_s(z;\eta)$ is introduced in \cite[equation (1.24)]{Krichever-90}.
It corresponds to a degeneration of the potential of the Lame equation.
The meromorphic function $\varphi_s(z)$ depends on the parameters $(\alpha,\kappa,\eta) \in \mathbb{C}^3$ with $\eta \ne 0$. For $\eta = 0$ set
\begin{equation} 
\varphi_s(z) =  z  \exp(\alpha z - \kappa z) / (1 - \kappa z).
\end{equation}

Set $Q(z) f(z) = z$ and $Q(z) = 1 + \sum_{k=1}^\infty q_k z^k$.
Comparing the series expansion of \eqref{fKrs} at zero with the expansion for $z/Q(z)$, we obtain
\begin{align} \label{qsk}
q_1 &= - \alpha, & 6 q_2 &= 3 \alpha^2 - 3 \kappa^2 + 2 \eta^2, & 6 q_3 &= - (\alpha - \kappa) (\alpha^2 + \kappa \alpha - 2 \kappa^2 + 2 \eta^2). 
\end{align}
\begin{lem}\label{L0}
The function \eqref{fKrs} with parameters $(\alpha,\kappa,\eta)$ is a solution of \eqref{feq} with \eqref{qsk}.
\end{lem}

The proof of this Lemma is a straightforward substitution.

For the singular Krichever function condition \eqref{det} never holds.
Moreover, for each series solution of \eqref{feq} determined by $(q_1, q_2, q_3, q_4)$
one can take either $\varphi(z)$ or $\varphi_s(z)$ with a choise of corresponding parameters to obtain a meromorphic solution with this series expansion at zero.
This assigns to any series solution of \eqref{feq} a solution in the class of meromorphic functions.
Let us note that for $\kappa = \eta$, $\eta \ne 0$ we have 
$\varphi_s(z) =  \exp(\alpha z) \sh(\eta z)/\eta$. Choosing $\alpha = (N-2) \eta /N$ we get a classical genus, see \cite[equation (1.27)]{Krichever-90}.

For small $n$ we find evidence for the following conjecture:

\begin{conj} \label{hyp}
 Any solution $f(z)$ of the Hirzebruch functional equation \eqref{Hfe} with $n>2$ and initial conditions $f(0) = 0$, $f'(0) = 1$ 
 is a solution of \eqref{feq}.
\end{conj}

\section{Todd function} \label{Sect5}

Let us consider the Todd function \eqref{ft} \vspace{-5pt}
\[
f(z) = {e^{a z} - e^{b z} \over a e^{b z} - b e^{a z}}.
\]
It depends on the parameters $(a,b) \in \mathbb{C}^2$, $a \ne b$.

Set $Q(z) f(z) = z$ and $Q(z) = 1 + \sum_{k=1}^\infty q_k z^k$.
Comparing the series expansion of \eqref{ft} at zero with the expansion for $z/Q(z)$, we obtain
\begin{align} \label{qt}
- 2 q_1 &= a + b, & 12 q_2 &= (a - b)^2, & q_3 &= 0, & - 720 q_4 &= (a-b)^4. 
\end{align}

\begin{lem}\label{L1}
The function \eqref{ft} with parameters $a, b$ satisfies the differential equation \eqref{feq}
\[
f(z) f'''(z) - 3 f'(z) f''(z) = 6 q_1 f'(z)^2 + 12 q_2 f(z) f'(z) + 12 q_3 f(z)^2,
\]
with initial conditions $f(0)=0$, $f'(0)=1$, where \eqref{qt}.
\end{lem}

The proof of Lemma \ref{L1} is a straightforward substitution. See \cite[Exercise E.5.11]{T}.
It is also a corollary of Lemma \ref{L0}, as the Todd function is a special case of \eqref{fKrs} with~$\alpha = \kappa$, $2 \kappa = a+b$, $2 \eta = a-b$.
See \cite[equation (1.26)]{Krichever-90}.

Denote by $M_0$ the irreducible two-dimensional algebraic manifold given by the relations $q_3 = 0$ and $5 q_4 = - q_2^2$
in $\mathbb{C}^4$ with coordinates $(q_1, q_2, q_3, q_4)$. Observe that for this parameters \eqref{det} does not hold.

For a Zariski open subset $U_0$ of $M_0$ given by $q_2 \ne 0$, a series solution to \eqref{feq} with such parameters coincides 
with the series expansion at zero for the Todd function with parameters $a, b$ that we obtain from the equations $- 2 q_1 = a + b$, $12 q_2 = (a - b)^2$.

\begin{dfn}
A \emph{series corresponding to the Todd function} is a series solution of \eqref{feq} with parameters in $M_0$.
\end{dfn}

By Theorem \ref{tft} and Corollary \ref{corU}, a series corresponding to the Todd function is a solution of Hirzebruch functional equation for any~$n$.
For $q_2 = 0$ the series corresponding to the Todd function coincides with the series expansion at zero of the rational function
\begin{equation} \label{q1}
 f(z) = {z \over 1 + q_1 z}. 
\end{equation}

\begin{cor}
Function \eqref{q1} is a solution of Hirzebruch functional equation for any~$n$.
\end{cor}

\section{Elliptic function of level $N$} \label{Sect6}

Let us consider the elliptic function of level $N$. Corollary \ref{CorKri} gives an expression for elliptic functions of level $N$ as Krichever 
functions with some specifications on the parameters. Let us prove that the existence of a lattice $L$ with periodicity properties~\eqref{ups}
characterizes elliptic functions of level $N$ among Krichever functions.

\begin{lem} \label{lemlemlem}
Let $f(z)$ be a Krichever function \eqref{fKr} with lattice $L = \langle \omega, \omega'\rangle$ and
\begin{align}
f(z + \omega) &= \exp\left(2 \pi I {k' \over N}\right) f(z), & f(z + \omega') &= \exp\left(- 2 \pi I {k \over N}\right) f(z),
\end{align}
where $(k, k', N) = 1$. Then $f(z)$ in an elliptic function of level $N$ with $\rho = {k \over N} \omega + {k' \over N} \omega'$.
\end{lem}

\begin{proof}
Comparing this periodicity properties with~\eqref{per} and using the Legendre identity~\eqref{Leg}
we obtain $\rho = {k \over N} \omega + {k' \over N} \omega'$ up to addition of a point in $L$.
Thus $\rho$ is an $N$-division point of $L$.
Consider the elliptic function of level $N$ with $\rho$.
By~Lemma~\ref{llN} in an appropriate lattice with $\omega = N \rho$ it has the same periodicity properties and the same poles as~$f(z)$,
and the initial conditions $f(0) = 0$, $f'(0) = 1$ imply that they concur.
\end{proof}

\begin{lem} \label{Lemmata} Any Krichever function~\eqref{fKr} with $g_2^3 \ne 27 g_3^2$ that is a solution of the Hirzebruch functional equation~\eqref{Hfe}
for some $n$ and some $c$
is an elliptic function of~level~$N \mid n$.
\end{lem}

\begin{proof}
Let $f(z)$ be a solution of \eqref{Hfe} for some $n$ and $c$ and let $f(z)$ have a periodicity property
$f(z+\omega) = \varepsilon f(z)$ for some $(\omega, \varepsilon) \in \mathbb{C}^2$ with $\omega \ne 0$. 
Compare the functional equation \eqref{Hfe} for $(z_1, z_2, \ldots, z_n)$ and $(z_1+\omega, z_2, \ldots, z_n)$.
Multiplying the second equation by $\varepsilon$ and subtracting the first equation, using the periodicity property we obtain
\begin{equation} 
(\varepsilon^{n} - 1) \prod_{j \ne 1} {1 \over f(z_j - z_1)} = (\varepsilon - 1) c.
\end{equation}
For a non-constant meromorphic function $f(z)$ this is possible only for $\varepsilon^{n} = 1$ and $(\varepsilon - 1) c = 0$.

The properties described in Section \ref{Sect3} give for a Krichever function \eqref{fKr} with $g_2^3 \ne 27 g_3^2$
a lattice $L = \langle\omega, \omega'\rangle$. We have just proved that $f(z+\omega) = \varepsilon f(z)$, $f(z+\omega') = \varepsilon' f(z)$
with~$\varepsilon^n = (\varepsilon')^n = 1$, thus Lemma \ref{lemlemlem} implies we get an elliptic function of level~$N \mid n$.
\end{proof}

\begin{cor}
 In the conditions of lemma \ref{Lemmata} we have $c = 0$.
\end{cor}

\begin{cor}[from Lemma \ref{L2} and Corollary \ref{CorKri}]
For any $N \geqslant 2$ the elliptic function of~level $N$ is a solution of \eqref{feq}
\[
f(z) f'''(z) - 3 f'(z) f''(z) = 6 q_1 f'(z)^2 + 12 q_2 f(z) f'(z) + 12 q_3 f(z)^2.
\]
\end{cor}

Again, set $Q(z) f(z) = z$ and $Q(z) = 1 + \sum_{k=1}^\infty q_k z^k$,
and use the coordinates $(q_1, q_2, q_3, q_4)$ as parameters for series solutions $f(z) = x + \sum_{k=1}^\infty f_k z^{k+1}$ of \eqref{feq}. Recall
\begin{align*}
f_1 &= - q_1, & f_2 &= q_1^2 - q_2, & f_3 &= -q_1^3 + 2 q_1 q_2- q_3, & f_4 &= q_1^4 - 3 q_1^2 q_2 + 2 q_1 q_3 + q_2^2 - q_4. 
\end{align*}

Denote by $U_N$ a subset of $\mathbb{C}^4$ with coordinates $(q_1, q_2, q_3, q_4)$ such that a series solution to \eqref{feq} with such parameters coincides 
with the series expansion at zero for the elliptic functions of level $N$ with some parameters.

Denote by $M_N$ the Zariski closure of $U_N$ in $\mathbb{C}^4$. As there exists an elliptic function of~level~$N$ for every lattice~$L$,
the algebraic manifold $M_N$ is not less than two-dimensional.
Suppose that we find a irreducible two-dimensional algebraic manifold in $\mathbb{C}^4$ that contains~$U_N$. This conditions imply that it coincides with $M_N$.

\begin{dfn} \label{defin}
A \emph{series corresponding to the elliptic function of level $N$} is a series solution of~\eqref{feq} with parameters in $M_N$.
\end{dfn}

We present a list of known series solutions of Hirzebruch functional equation \eqref{Hfe} in Table~\ref{tabl}.
The data for elliptic functions of level $2$, $3$ and $4$ is taken from \cite[Sections~5,~6, and~7]{Ell4}.
We do not specify the expression for $c$ in $q_i$ in the first row as it depends on $n$.

\begin{table}[ht]
\caption{\label{list}Previously known solutions of Hirzebruch functional equation \eqref{Hfe}.}
\begin{center}
\begin{tabular}{|c|c|c|c|c|}
\hline
$n$ & $c$ & N & equations for $M_N$ & series corresponding to \\
\hline
\hline
any & $ \star $ & 0 & $q_3 = 0$, \hspace{65pt} $5 q_4 = - q_2^2$ \hfill \text{} & Todd function\\
\hline
$2 \mid n$ & 0 & 2 & $q_1 = 0$, \hspace{70.6pt} $q_3 = 0$ \hfill \text{} &  elliptic of level $2$\\
\hline
$3 \mid n$ & 0 & 3 & $q_2 = -q_1^2$, \hspace{51pt} $5 q_4 = - q_1 (4 q_3 + q_1^3)$ \hfill \text{} &  elliptic of level $3$ \\
\hline
$4 \mid n$ & 0 & 4 & $q_3 = - q_1 (q_1^2 + 3 q_2)$, $10 q_4 = q_1^4 + 6 q_1^2 q_2 + 7 q_2^2$ &  elliptic of level $4$\\
\hline
\end{tabular}
\label{tabl}
\end{center}
\end{table}

The classification result that all the solutions of the Hirzebruch functional equation \eqref{Hfe} for $n = 3$ and $n = 4$ are given 
in Table 1 will be proved in Sections~\ref{Clas3} and \ref{Clas4}.
It follows from Table \ref{tabl} that we have $\mathcal{M}_3 = M_0 \cup M_3$ and 
$\mathcal{M}_4 = M_0 \cup M_2 \cup M_4$.

\section{Classification for $n = 3$.} \label{Clas3}
\begin{thm}
Any series solution of the Hirzebruch functional equation for $n = 3$
corresponds to the Todd function or to the elliptic function of level $3$.
\end{thm}

\begin{proof}
The Hirzebruch functional equation \eqref{Hfe} for $n = 3$ in $Q(z) = z / f(z)$ is (see \eqref{fe2})
\[
 {Q(z_2 - z_1) Q(z_3 - z_1) \over (z_2 - z_1) (z_3 - z_1)} + {Q(z_1 - z_2) Q(z_3 - z_2) \over (z_1 - z_2) (z_3 - z_2)}
 + {Q(z_1 - z_3) Q(z_2 - z_3) \over (z_1 - z_3) (z_2 - z_3)} = c.
\]
By expanding this relation up to $z^6$ we get the expressions
\begin{align*}
&c = 3 (q_1^2 + q_2), \qquad - 5 q_4 = q_2^2 + 4 q_1 q_3, \qquad q_5 = q_2 q_3, \qquad 35 q_6 = 2 q_2^3 - 12 q_1 q_2 q_3 - 5 q_3^2,\\
&5 q_7 = 2  q_3 (q_2^2 - q_1 q_3), \qquad - 525 q_8 = 9 q_2^4 + 72 q_1 q_2^2 q_3 - 56 q_1^2 q_3^2 + 100 q_2 q_3^2,
\end{align*}
and the relation $(q_1^2 + q_2) q_3^2 = 0$. Thus $c = 0$ or $q_3 = 0$.

By Theorem \ref{TM} each solution is determined by the values $q_1$, $q_2$, $q_3$. 
For $c = 0$ we observe $q_2 = -q_1^2$ and get a series corresponding to the elliptic function of~level~$3$.
For~$q_3 = 0$ we get a series corresponding to the Todd function. See Table \ref{tabl}.
\end{proof}

\section{Classification for $n = 4$.} \label{Clas4}
\begin{thm}
Any series solution of the Hirzebruch functional equation for $n = 4$
corresponds to the Todd function or to the elliptic function of level~$2$~or~$4$.
\end{thm}

\begin{proof}
The Hirzebruch functional equation \eqref{Hfe} for $n = 4$ in $Q(z) = z / f(z)$ is (see \eqref{fe2})
\begin{multline*}
{Q(z_2 - z_1) Q(z_3 - z_1) Q(z_4 - z_1) \over (z_2 - z_1) (z_3 - z_1) (z_4 - z_1)} + {Q(z_1 - z_2) Q(z_3 - z_2) Q(z_4 - z_2) \over (z_1 - z_2) (z_3 - z_2) (z_4 - z_2)} + \\
 + {Q(z_1 - z_3) Q(z_2 - z_3) Q(z_4 - z_3) \over (z_1 - z_3) (z_2 - z_3) (z_4 - z_3)} + {Q(z_1 - z_4) Q(z_2 - z_4) Q(z_3 - z_4) \over (z_1 - z_4) (z_2 - z_4) (z_3 - z_4)}= c.
\end{multline*}
Expanding this equation up to~$z^6$ 
we get $c = 4 (q_1^3 + 3 q_1 q_2 + q_3)$, expressions for $q_5, q_6, q_7, q_8, q_9$, 
and the relations
\begin{align} 
&3 c (q_1^2 q_2 - 3 q_2^2 + q_1 q_3) = 4 (q_1^3 - 3 q_1 q_2 + 3 q_3) (3 q_1^2 q_2 + 7 q_2^2 - q_1 q_3 - 10 q_4), \label{I41} \\
&c q_3^2 = 0, \label{I43} \\
&(4 q_2^2 + q_1 q_3 + 20 q_4) ( 2 q_1 q_2^2 + q_1^2 q_3 + 3 q_2 q_3 + 10 q_1 q_4) = 0.
 \label{I42}
\end{align}

From \eqref{I43} we have $c = 0$ or $q_3 = 0$. Let us consider this two cases:
\begin{enumerate}
 \item Case $c = 0$, that is $q_3 = - q_1 (q_1^2 + 3 q_2)$. 
Relation \eqref{I41} becomes
\[
 q_1 (q_1^2 + 6 q_2) (q_1^4 + 6 q_1^2 q_2 + 7 q_2^2 - 10 q_4) = 0.
\]
For $6 q_2 = - q_1^2$ we have $36 (q_1^4 + 6 q_1^2 q_2 + 7 q_2^2 - 10 q_4) = 7 q_1^4 - 360 q_4$ and relation \eqref{I42} becomes
$
q_1 (7 q_1^4 - 360 q_4)^2 = 0, 
$
thus we conclude that
\[
 q_1 (q_1^4 + 6 q_1^2 q_2 + 7 q_2^2 - 10 q_4) = 0.
\]
For $q_1 = 0$ we observe $q_3 = 0$ and get a series corresponding to the elliptic function of level $2$.
For $10 q_4 = q_1^4 + 6 q_1^2 q_2 + 7 q_2^2$ we observe $q_3 = - q_1 (q_1^2 + 3 q_2)$ and get a series corresponding to the elliptic function of level $4$.
\item Case $q_3 = 0$. 
We have $c = 4 q_1 (q_1^2 + 3 q_2)$ and relation \eqref{I42} becomes
\[
q_1 (q_2^2 + 5 q_4)^2 = 0.
\]
Therefore either $c = 0$, and we are in case (1) above, or $5 q_4 = - q_2^2$,
and we get a series corresponding to the Todd function.
\end{enumerate}

\vspace{-14pt}

\end{proof}

\section{Classification for $n = 5$.} \label{Clas5}

\begin{thm} \label{T61}
A series corresponding to the elliptic function of level $5$ is a solution to~\eqref{feq} with
parameters $(q_1, q_2, q_3, q_4)$ in the two-dimensional irreducible algebraic manifold $M_5$ determined in $\mathbb{C}^4$ by the equations
\begin{align} 
&q_1^4 + 6 q_1^2 q_2 + 2 q_2^2 + 4 q_1 q_3 + q_4 = 0, \label{5q4}\\ 
&(11 q_1^2 + 3 q_2)^3 - 4 (18 q_1^3 + 6 q_1 q_2 - q_3)^2 = 0. \label{5q123}
\end{align}
\end{thm}

We will obtain the proof of this theorem while proving Theorem \ref{T62}.

Denote by $P_5$ the expression $5 q_1^5 + 30 q_1^3 q_2 + 9 q_1 q_2^2 + 22 q_1^2 q_3 + q_2 q_3 + 5 q_5$, and by $P_6$ the left hand side of \eqref{5q123}.

\begin{thm} \label{T62}
Any series solution of the Hirzebruch functional equation for $n = 5$
corresponds to the Todd function or to the elliptic function of level $5$.
\end{thm}

\begin{proof}
Expanding the Hirzebruch functional equation \eqref{Hfe} for $n = 5$ in $Q(z) = z / f(z)$ up to~$z^6$ 
we get $c = 5 (q_1^4 + 6 q_1^2 q_2 + 2 q_2^2 + 4 q_1 q_3 + q_4)$, the expressions for $q_6, q_7, q_8, q_9, q_{10}$, and the relations:
\begin{multline} \label{I51}
12 (q_1^3 - 3 q_1 q_2 + 3 q_3) P_5 - (q_1^2 - 27 q_2) P_6 = c (5 q_1^4 + 150 q_1^2 q_2 + 77 q_2^2 + 36 q_1 q_3 - 20 q_4),
\end{multline}
\begin{multline} \label{I52}
(q_1^4 - 6 q_1^2 q_2 + 7 q_2^2 + 4 q_1 q_3 - 10 q_4) P_5 - 3 (q_1^3 + 5 q_1 q_2 + 4 q_3) P_6 = \\
= - 2 c (10 q_1^5 + 60 q_1^3 q_2 + 17 q_1 q_2^2 + 41 q_1^2 q_3 + 15 q_2 q_3 - 5 q_1 q_4),
\end{multline}
\begin{multline} \label{I53}
5 (23 q_2^2 + 4 q_1 q_3 - 29 q_4) P_6 = \\
= 4 c (202 q_1^2 q_2^2 + 78 q_2^3 + 36 q_1^3 q_3 + 12 q_1 q_2 q_3 + 59 q_3^2 - 250 q_1^2 q_4 - 150 q_2 q_4),
\end{multline}
\begin{multline} \label{I54}
4 (39 q_1^5 + 243 q_1^3 q_2 + 94 q_1 q_2^2 + 138 q_1^2 q_3 - 35 q_2 q_3 + 110 q_1 q_4 - 40 q_5) P_5 +\\
+ (132 q_1^4 + 781 q_1^2 q_2 + 189 q_2^2 + 544 q_1 q_3) P_6 + 4 c^2 (32 q_1^2 + 17 q_2) = \\
= c (1720 q_1^6 + 11015 q_1^4 q_2 + 7146 q_1^2 q_2^2 + 1179 q_2^3 + 7088 q_1^3 q_3 + 2572 q_1 q_2 q_3 - 48 q_3^2).
\end{multline}

Denote by $I$ the ideal generated by the relations \eqref{I51}--\eqref{I54}.

Consider the cases:
\begin{enumerate}
 \item Case $c = 0$, that is \eqref{5q4}. In this case in $I$ there are the polynomials $P_5^2$ and $P_6^3$.
From $P_5 = 0$ we obtain an expression for $q_5$, while $P_6 = 0$ is \eqref{5q123}.

The elliptic function of level $5$ is two-parametric and solves the Hirzebruch functional equation \eqref{Hfe} for $n=5$ and $c = 0$,
therefore we have proved Theorem~\ref{T61}.
We obtain a series corresponding to the elliptic function of level $5$.

\item Case $q_3 = 0$. In this case we have $c q_5^3 \in I$ and $c (q_2^2 + 5 q_4)^3 \in I$.
The case $c = 0$ has been considered above. In the remaining case $q_5 = 0$ and $5 q_4 = - q_2^2$.
We obtain a series corresponding to the Todd function.

\item Case $c \ne 0$, $q_3 \ne 0$. In the general case, we have:
\begin{align*}
&c^2 q_3^5 (3 \cdot 67^2 q_2 - 1321 q_1^2 ) \in I,  &
&c^2 q_3^5 (9 \cdot 67^4 q_4 - 1122211 q_1^4) \in I,  \\
&c^2 q_3^5 (67^3 q_3 + 4752 q_1^3) \in I,  &
&c^2 q_3^5 (5 \cdot 67^5 q_5 + 19282032 q_1^5) \in I. 
\end{align*}
Thus we obtain expressions for $q_2, q_3, q_4, q_5$ in $q_1$.
They satisfy \eqref{I51}--\eqref{I54} for any~$q_1$.
However, the series expansion of the Hirzebruch functional equation up to~$z^{7}$ gives~$q_1 = 0$, and in this case we get only the trivial solution $f(z) = z$. 
\end{enumerate} \vspace{-14pt}

\end{proof}

\section{Classification for $n = 6$.} \label{Clas6}

Denote
\begin{align*}
&P_5 = 5 q_1^5 + 50 q_1^3 q_2 + 51 q_1 q_2^2 + 48 q_1^2 q_3 + 24 q_2 q_3 + 30 q_1 q_4,\\
&P_6 = (13 q_1^3 + 9 q_1 q_2 + 3 q_3)^2 - 162 q_1^4 (q_1^2 + q_2),\\
&P_7 = 56 q_1^5 q_2 + 576 q_1^3 q_2^2 + 648 q_1 q_2^3 - 15 q_1^4 q_3 + 474 q_1^2 q_2 q_3 + 279 q_2^2 q_3 - 144 q_1 q_3^2 - 90 q_3 q_4.
\end{align*}
We have $3 q_3 P_5 - 8 q_2 P_6 + q_1 P_7 = 0$. The algebraic manifold in $\mathbb{C}^4$ determined by 
$P_5 = P_6 = P_7 = 0$
has two irreducible components.
One of them is $M_2$. Denote the other one by~$A$. Denote
\begin{align*}
&Q_6 = 4 q_1^3 q_3 - 2 q_1^2 q_2^2 - 10 q_1^2 q_4 + 12 q_1 q_2 q_3 - 5 q_2^3 - 15 q_2 q_4 + 5 q_3^2 + 35 q_6.
\end{align*}

\begin{thm} \label{tete}
A series corresponding to the elliptic function of level $6$ is a solution to~\eqref{feq} with parameters $(q_1, q_2, q_3, q_4)$ in $A$, that is $A = M_6$.
\end{thm}

We will obtain the proof of this theorem while proving Theorem \ref{Tn6}.

\begin{thm} \label{Tn6}
Any series solution of the Hirzebruch functional equation for $n = 6$ with~$c = 0$
corresponds to the elliptic function of level $2$, $3$ or $6$.
\end{thm}

\begin{proof}
Expanding the Hirzebruch functional equation for $n = 6$, $c = 0$, in $Q(z)=~z / f(z)$ up to~$z^6$ 
we get 
\begin{align}
q_5 &= - (q_1^5 + 10 q_1^3 q_2 + 10 q_1 q_2^2 + 10 q_1^2 q_3 + 5 q_2 q_3 + 5 q_1 q_4), 
\end{align}
expressions for $q_7, q_8, q_9, q_{10}, q_{11}$, relations on $q_1, q_2, q_3, q_4, q_6$ including
\begin{multline} \label{I6A}
(107 q_1^4 - 136 q_1^2 q_2  + 75 q_2^2 + 396 q_1 q_3 + 150 q_4) P_5 + 20 q_1 (q_1^2 - 7 q_2) P_6 - 20 q_2 P_7 = \\
= 9 (q_1^3 - 3 q_1 q_2 + 3 q_3) Q_6, 
\end{multline}
\begin{multline} \label{I6B}
2 (341 q_1^5 + 3697 q_1^3 q_2 + 4140 q_1 q_2^2 + 3618 q_1^2 q_3 + 3630 q_2 q_3) P_5 + \\
+ 20 (11 q_1^4 + 123 q_1^2 q_2 + 112 q_2^2 + 120 q_1 q_3) P_6 + \\
+ 300 q_3 P_7 = 3 (3 q_1^4 - 39 q_1^2 q_2 + 70 q_2^2 + 19 q_1 q_3 - 100 q_4) Q_6,
\end{multline}
and three more relations that we skip for brevity. 
Denote by $I$ the ideal generated by all five relations. 
We have
\begin{align*}
 P_5^4 &\in I, & Q_6 P_6^3 &\in I, &  (8 q_1 P_6 + P_7)^4 &\in I.
\end{align*}
Therefore we have $P_5 = 0$, $P_7 = - 8 q_1 P_6$ and $(q_2 + q_1^2) P_6 = 0$.
Consider the cases:
\begin{enumerate}
 \item Case $q_1 = q_3 = 0$. Equation \eqref{I6B} becomes $(7 q_2^2 - 10 q_4) Q_6 = 0$.
 For~$Q_6 = 0$ we obtain a series corresponding to~the elliptic function of level $2$. 
 For $10 q_4 = 7 q_2^2$ observe that we get the intersection of initial conditions for series corresponding to elliptic functions of level $2$ and $4$. 
 We have $I = 0$, the series expansion of the Hirzebruch functional equation up to~$z^{7}$ gives $70 q_6 = 31 q_2^3$, therefore again $Q_6 = 0$ and
 we obtain a series corresponding to~the elliptic function of level $2$.
 \item
 Case $q_2 = - q_1^2$ and $Q_6 = 0$. From $P_5 = 0$ we get $q_1 (q_1^4 + 4 q_1 q_3 + 5 q_4) = 0$. Thus either $5 q_4 = - q_1 (q_1^3 + 4 q_3)$,
 and we get a series corresponding to an elliptic function of level $3$, or $q_1 = q_2 = q_5 = 0$, $7 q_6 = - q_3^2$.
 In the last case \eqref{I6B} implies~$q_3 q_4 = 0$.
 For~$q_3 = 0$ see case (1) above.
 For $q_4 = 0$ we get a series corresponding to the elliptic function of level $3$.
 \item Case $P_5 = P_6 = P_7 = 0$, $q_1 \ne 0$.  
 From \eqref{I6A} and \eqref{I6B} we get
 \begin{align}
 \qquad \quad (q_1^3 - 3 q_1 q_2 + 3 q_3) Q_6 &= 0, &  (3 q_1^4 - 39 q_1^2 q_2 + 70 q_2^2 + 19 q_1 q_3 - 100 q_4) Q_6 &= 0. \label{leq}
 \end{align}
 We show that in this conditions $Q_6 \ne 0$ is impossible.
 If not, from~\eqref{leq} we obtain $3 q_3 = - q_1 (q_1^2 - 3 q_2)$, $- 30 q_4 = q_1^4 + 6 q_1^2 q_2 - 21 q_2^2$.
 From $P_5 = 0$ we get~$q_1 (q_1^2 + q_2) (q_1^2 - 8 q_2) = 0$. Recall $q_1 \ne 0$.
 In the cases $q_2 = - q_1^2$ and $8 q_2 = q_1^2$ the series expansions of the Hirzebruch functional equation \eqref{Hfe} up to~$z^{7}$ give the
 relations $315 q_6 = - 242 q_1^6$ and $322560 q_6 = 3751 q_1^6$, accordingly. Thus $Q_6 \ne 0$ is impossible.
 
 Therefore $Q_6 = 0$, and we have
 proved Theorem \ref{tete}. In this case we get a series corresponding to the elliptic function of level $6$.

\end{enumerate} \vspace{-14pt}

\end{proof}

\begin{thm} \label{Tn6c1}
Any series solution of the Hirzebruch functional equation for $n = 6$ 
corresponds to the Todd function or to the elliptic function of level $2$, $3$ or $6$.
\end{thm}

\begin{proof}

Expanding the Hirzebruch functional equation for $n = 6$ in $Q(z)=~z / f(z)$ up to~$z^7$ 
we get \vspace{-6pt}
\begin{align}
c &= 6 (q_1^5 + 10 q_1^3 q_2 + 10 q_1 q_2^2 + 10 q_1^2 q_3 + 5 q_2 q_3 + 5 q_1 q_4 + q_5), 
\end{align}
expressions for $q_7, q_8, q_9, q_{10}, q_{11}, q_{12}$, and relations on $q_1, q_2, q_3, q_4, q_5, q_6$.
Denote by $I$ the ideal generated by this relations.

We have $c^2 q_3^5 (349 q_1^2 - 4232 q_2) \in I$.
Therefore $c =0$, or $q_3 = 0$, or $4232 q_2 = 349 q_1^2$.
The case $c=0$ has been considered in Theorem \ref{Tn6}. Suppose $c \ne 0$. We obtain two cases:
\begin{enumerate}
 \item Case $q_3 = 0$. In this case we have $c q_5^3 \in I$, therefore $q_5 = 0$, $c^3 (q_2^2 + 5 q_4)^3 \in I$, therefore $5 q_4 = - q_2^2$,
 and $c (q_2^3 + 3 q_2 q_4 - 7 q_6)^3 \in I$, therefore $7 q_6 = q_2^3 + 3 q_2 q_4$. We obtain a series corresponding to the Todd function.
 \item Case $4232 q_2 = 349 q_1^2$. In this case we have 
 \begin{align*}
 &\hspace{24pt} c^2 q_3^2 (7^3 \cdot 13 q_1^3 + 2^3 \cdot 23^3 \cdot 3 q_3) \in I, & &c^2 q_3^2 (7^3 \cdot 13 \cdot 1643 q_1^5 + 2^7 \cdot 23^5 \cdot 3 q_5) \in I, \\
 &\hspace{24pt} c^2 q_3^2 (3304589 q_1^4 - 2^7 \cdot 23^4 \cdot 15 q_4) \in I, & &c^2 q_3^2 (76731365059 q_1^6 - 2^{10} \cdot 23^6 \cdot 315 q_6) \in I.
 \end{align*}
 Therefore either $q_3 = 0$, and we are in case (1), or we obtain expressions for $q_3, q_4, q_5, q_6$ in $q_1$.
 They give $I=0$ for any~$q_1$.
However, the series expansion of the Hirzebruch functional equation up to~$z^{8}$ gives~$q_1 = 0$,
and in this case we get only the trivial solution $f(z) = z$. 
\end{enumerate} \vspace{-14pt}

\end{proof} \vspace{-4pt}

\section{Meromorphic functions solutions}

\begin{cor}
Any solution of the Hirzebruch functional equation \eqref{Hfe} with $3 \leqslant n \leqslant 6$ 
and initial conditions $f(0)=0, f'(0)=1$
in the class of meromorphic functions is one of~the following:
\begin{itemize}
 \item The Todd function \eqref{ft} with $(-1)^n (a-b) c = - (a^n - b^n)$ and $a \ne b$.
 \item The rational function $z / (1 + q_1 z)$ with $c = n q_1^{n-1}$.
 \item The elliptic function of level $N$ with $N \mid n$ and $c = 0$.
 \item The function $\exp(\alpha z) \sh(\eta z)/\eta$ with $N \alpha = (N-2 k) \eta$ for $k = 1,2, \ldots, [N/2] $ and~$N \mid n$, $\eta \ne 0$,~$c= 0$.
\end{itemize}
\end{cor}

The proof is given troughout the work. In the last case, we intersect $M_N$ where $2 \leqslant N \leqslant 6$
with the compliment to \eqref{det} to get a one-dimensional manifold.
We check directly that it's components not covered by the first two cases belong to the last one.

\section*{Acknowledgements}
 
The author is a Young Russian Mathematics award winner and would like to thank its~sponsors and~jury.

\end{document}